\documentclass[11pt]{amsart}
\usepackage[colorlinks=true,pagebackref,hyperindex,citecolor=blue,linkcolor=red]{hyperref}
\usepackage{appendix}
\usepackage{amsmath}
\usepackage{amsfonts}
\usepackage{amssymb}
\usepackage{color}
\usepackage[all]{xy}  
\usepackage{enumerate}
\usepackage[top=1in, bottom=1in, left=1in, right=1in]{geometry}
\usepackage{mathrsfs}

\usepackage{stmaryrd}
\usepackage{rotating}

\usepackage[english]{babel}

\theoremstyle{definition}
\newtheorem{theorem}{Theorem}[section]

\newtheorem{question}[theorem]{Question}
\newtheorem{corollary}[theorem]{Corollary}

\newtheorem{proposition}[theorem]{Proposition}

\newtheorem*{ex}{Example}

\theoremstyle{definition}
\newtheorem{definition}[theorem]{Definition}

\newtheorem{example}[theorem]{Example}
\newtheorem*{disc}{Discussion}

\newtheorem{remark}[theorem]{Remark}
\numberwithin{equation}{subsection}

\newcommand{\m}{\mathfrak{m}}

\newcommand{\NN}{\mathbb{N}}
\newcommand{\ZZ}{\mathbb{Z}}
\newcommand{\QQ}{\mathbb{Q}}

\newcommand{\Depth}{\operatorname{depth}}

\newcommand{\Ext}{\operatorname{Ext}}
\newcommand{\Tor}{\operatorname{Tor}}

\newcommand{\Char}{\operatorname{char}}

\newcommand{\Ht}{\operatorname{ht}}

\newcommand{\ls}{\leqslant}
\newcommand{\gs}{\geqslant}
\newcommand{\ds}{\displaystyle}

\newcommand{\ov}[1]{\overline{#1}}

\newcommand{\grade}{\operatorname{grade}}

\begin{document}
\newcommand{\tens}{\otimes}
\newcommand{\hhtest}[1]{\tau ( #1 )}
\renewcommand{\hom}[3]{\operatorname{Hom}_{#1} ( #2, #3 )}
\newcommand{\ind}{\operatorname{index}}
\newcommand{\gll}{\operatorname{g\ell\ell}}
\newcommand{\ord}{\operatorname{ord}}
\renewcommand{\ll}{\operatorname{\ell\ell}}
\newcommand{\soc}{\operatorname{soc}}
\newcommand{\CCC}{\mathfrak{C}}
\newcommand{\frk}{\operatorname{f-rank}}
\newcommand{\lcm}{\operatorname{lcm}}
\newcommand{\PPi}{\ds \left(\Pi\right)}
\newcommand{\MinGen}{\operatorname{MinGen}}

\title{Products of ideals may not be Golod}
\author{Alessandro De Stefani}

\begin{abstract} We exhibit an example of a product of two proper monomial ideals such that the residue class ring is not Golod. We also discuss the strongly Golod property for rational powers of monomial ideals, and introduce some sufficient conditions for weak Golodness of monomial ideals. Along the way, we ask some related questions.
\end{abstract}
\maketitle
\section{Introduction}
Let $k$ be a field, and let $(R,\m,k)$ denote a Noetherian positively graded $k$-algebra, with $R_0=k$ and irrelevant maximal ideal $\m=\bigoplus_{i \gs 1} R_i$. Consider the Poincar{\'e} series of $R$
\[
\ds P_R(t) = \sum_{i \gs 0} \dim_k \Tor_i^R(k,k)t^i
\]
which is, in general, not rational \cite{Anick}. If $n = \dim_k \m/\m^2$ is the embedding dimension of $R$, Serre showed that $P_R(t)$ is bounded above term by term by the following rational series
\[
\ds \frac{(1+t)^n}{\ds 1-t\sum_{i \gs 1} \dim_k \left(H_i(R)\right)t^i}.
\]
Here $H_i(R)$ is the $i$-th homology of the Koszul complex on a minimal homogeneous generating set of $\m$, over the ring $R$. The ring $R$ is called {\it Golod} if equality holds. As a consequence, Golod rings have rational Poincar{\'e} series. The main purpose of this article is to answer, in negative, the following question
\begin{question}\cite[Problem 6.18]{McCPeeva} \label{mainq} Let $k$ be a field, and let $(R,\m,k)$ be positively graded $k$-algebra. Let $I,J$ be two proper homogeneous ideals in $R$. Is the ring $R/IJ$ always Golod?
\end{question}
As reported in \cite{McCPeeva}, Question \ref{mainq} was first asked by Volkmar Welker. The general belief, supported by strong computational evidence, was that this question had positive answer. The first result in this direction is a theorem of Herzog and Steurich \cite{HerSt}: let $S$ be a polynomial ring over a field, and let $I$,$J$ be two proper homogeneous ideals of $S$. If $I\cap J = IJ$, then $S/IJ$ is Golod. Another reason to believe that Question \ref{mainq} had positive answer comes from a result of Avramov and Golod \cite{AvGld}, which says that Golod rings are never Gorenstein, unless they are hypersurfaces. This is consistent with a result of Huneke \cite{HunProd}, according to which $S/IJ$ is never Gorenstein, unless $I$ and $J$ are principal. More recently, Herzog and Huneke show that, if $I$ is a homogeneous ideal in a polynomial ring $S$ over a field of characteristic zero, then, for all $d \gs 2$, the ring $S/I^d$ is Golod \cite[Theorem 2.3 (d)]{HerHun}. In \cite[Theorem 1.1]{FkhWel} Seyed Fakhari and Welker write that any product of proper monomial ideals in a polynomial ring over a field is Golod. The key step in their proof is to show that products of monomial ideals always satisfy the strong-GCD condition. This condition is the existence of a linear order on a minimal monomial generating set of the ideal, satisfing certaintain properties \cite[Definition 3.8]{Joll}. The fact that monomial ideals that satisfy the strong-GCD condition are Golod is first stated by J{\"o}llenbeck in \cite[Theorem 7.5]{Joll}, provided an extra assumption, called Property (P), is satisfied, and then by Berglund and J{\"o}llenbeck in \cite[Theorem 5.5]{BerJoll}, where the extra assumption is removed. 

In Section \ref{Sec_counter}, we provide examples of products of proper monomial ideals in a polynomial ring $S$ over a field, such that the residue class ring is not Golod. For instance, Example \ref{Sri}:
\begin{ex} Let $k$ be a field, and let $S=k[x,y,z,w]$ be a polynomial ring, with standard grading. Consider the monomial ideals $\m = (x,y,z,w)$ and $J = (x^2,y^2,z^2,w^2)$ inside $S$. Let 
\[
\ds I:=\m J = (x^3,x^2y,x^2z,x^2w,xy^2,y^3,y^2z,y^2w,xz^2,yz^2,z^3,z^2w,xw^2,yw^2,zw^2,w^3)
\]
be their product, and set $R=S/I$. Then, the ring $R$ is not Golod.

\end{ex} 
Our example satisfies the strong-GCD condition. Indeed, the argument of \cite[Theorem 1.1]{FkhWel} is correct, but it only shows that products of monomial ideals satisfy the strong-GCD condition. We have not been able to locate specifically where the mistake in \cite{Joll} or \cite{BerJoll} may be.

In Section \ref{frac} we study the strongly Golod property for rational powers of monomial ideals. Let $S = k[x_1,\ldots,x_n]$ be a polynomial ring over a field $k$ of characteristic zero, and let $I \subseteq S$. In \cite{HerHun} Herzog and Huneke introduce the following notion: $I$ is called {\it strongly Golod} if $\partial(I)^2 \subseteq I$, where $\partial(I)$ is the ideal of $S$ generated by the partial derivatives of elements in $I$. The main point of this definition is that, if an ideal $I$ is strongly Golod, then the ring $S/I$ is Golod  \cite[Theorem 1.1]{HerHun}. Among other things, in Section \ref{frac} we show that if $I$ is a strongly Golod monomial ideal, then so is $I^{p/q}$, for any $p \gs q$. This generalizes \cite[Proposition 3.1]{HerHun}. 

It is easy to find examples of ideals that are Golod, but not strongly Golod. In \cite{HerHun}, Herzog and Huneke introduce the notion of {\it squarefree strongly Golod} ideal, that applies to squarefree monomial ideals. This is a weakening of the strongly Golod definition, but it still implies that the multiplication on the Koszul homology is identically zero. We will say that a ring is {\it weakly Golod} if the multiplication on Koszul homology is trivial. In \cite[Theorem 5.1]{BerJoll}, Berglund and J{\"o}llenbeck show that, in case the ideal in question is monomial, weak Golodness and Golodness are equivalent notions. Herzog and Huneke use this result in \cite[Theorem 3.5]{HerHun} to conclude that squarefree strongly Golod ideals are Golod. See Section \ref{quests} for more discussions and questions about this topic. In Section \ref{lcm} we introduce {\it $\lcm$-strongly Golod} monomial ideals, which are a more general version of squarefree strongly Golod ideals. We show that $\lcm$-strongly Golod ideals are weakly Golod. In Section \ref{quests}, we give some sufficient conditions for an ideal to be strongly Golod, and we ask several related questions. In Appendix \ref{res} we record a minimal free resolution for Example \ref{main}, for convenience of the reader. All computations are made using the computer software system Macaulay2 \cite{Mac2}.

\section{Examples of products that are not Golod} \label{Sec_counter}
Golod rings were named after Evgenii S. Golod, who proved that the upper bound in Serre's inequality is achieved if and only if the Eagon resolution is minimial \cite{Golod}. This happens if and only if all the Massey operations of the ring vanish. Since the vanishing of the second Massey operation means that every product of Koszul cycles of positive homological degree is a boundary, Golod rings have, in particular, trivial multiplication on the positive degree Koszul homology. We will use this fact in the proofs of our examples in this section. See \cite[Chapter 4]{GullLev} or \cite[Section 5.2]{AvrInf} for details and more general statements. 

If $(R,\m,k)$ is a Noetherian positively graded algebra over a field $k$, we can write $R \cong S/I$, where $S=k[x_1,\ldots,x_n]$ is a polynomial ring, and $I \subseteq S$ is a homogeneous ideal. If $\m = (x_1,\ldots,x_n)$ denotes the irrelevant maximal ideal of $S$, we can always assume that $I \subseteq \m^2$. Let $K_\bullet$ be the Koszul complex on the elements $x_1,\ldots,x_n$ of $S$, which is a minimal free resolution of $k$ over $S$. We have that $K_1$ is a free $S$-module of rank $n$, and we denote by $\{e_{x_1},\ldots,e_{x_n}\}$ a basis. In addition, we have that $K_i \cong \bigwedge_i K_1$ for all $i=1,\ldots, n$, and the differential $\delta_i:K_i \to K_{i-1}$ on a basis element is given by
\[
\ds \delta_i(e_{t_1} \wedge \ldots \wedge e_{t_i}) = \sum_{j=1}^i (-1)^{j-1} \ t_j \ e_{t_1} \wedge \ldots \wedge e_{t_{j-1}} \wedge e_{t_{j+1}} \wedge \ldots \wedge e_{t_i},
\]
and extended by linearity to $K_i$. Let $K_\bullet(R) = K_\bullet \otimes_S R$ be the Koszul complex on $R$. We denote by $Z_\bullet(R)$ the Koszul cycles, and by $H_\bullet(R)$ the Koszul homology on $R$. 

We are now ready for the first example.

\begin{example} \label{Sri} Let $k$ be a field, and let $S=k[x,y,z,w]$, with the standard grading. Let $\m=(x,y,z,w)$ be the irrelevant maximal ideal, consider the monomial ideal $J = (x^2,y^2,z^2,w^2)$ and let
\[
\ds I:=\m J = (x^3,x^2y,x^2z,x^2w,xy^2,y^3,y^2z,y^2w,xz^2,yz^2,z^3,z^2w,xw^2,yw^2,zw^2,w^3).
\]
Then, the ring $R = S/I$ is not Golod.
\end{example}
\begin{proof}
Golod rings have trivial multiplication on $H_\bullet(R)_{\gs 1}$. Therefore, to show that $R$ is not Golod, it is enough to show that there exist two elements $\alpha,\beta \in H_\bullet(R)_{\gs 1}$ such that $\alpha \beta \ne 0$. Consider the element $u=(e_x\wedge e_y) \otimes xy \in K_2(R)$. It is a Koszul cycle:
\[
\ds \delta_2(u) = e_y \otimes x^2y -e_x \otimes xy^2 =0 \ \mbox{ in } K_1(R),
\]
because $x^2y \in I$ and $xy^2 \in I$. Then, let $\alpha := [u] \in H_2(R)$ be its residue class in homology. Similarly, let $v=(e_z \wedge e_w) \otimes zw \in Z_2(R)$, and let $\beta:= [v] \in H_2(R)$. We want to show that $uv = (e_x \wedge e_y \wedge e_z \wedge e_w) \otimes xyzw \in Z_4(R)$ is not a boundary, so that $\alpha\beta = [uv] \ne 0$ in $H_4(R)$. Note that $K_5(R) = 0$, hence such a product is zero in homology if and only if $xyzw \in I$. But $xyzw \notin I$, as every monomial generator of $I$ contains the square of a variable. 
\end{proof}
\begin{remark} We keep the same notation as in Example \ref{Sri}. Using Macaulay2 \cite{Mac2}, one can compute the first Betti numbers of $k$ over $R$:
\[
\xymatrixcolsep{5mm}
\xymatrixrowsep{2mm}
\xymatrix{
\ldots \ar[r] & R^{11283} \ar[r] & R^{2312} \ar[r] & R^{493} \ar[r] & R^{98} \ar[r] & R^{22} \ar[r] & R^4 \ar[r] & R \ar[r] & k \ar[r] & 0.
}
\]
Therefore the Poincar{\'e} series of $R$ is 
\[
\ds P_R(t) = 1+4t+422t^2+98t^3 + 493t^4 + 2312t^5 + 11283t^6 + \ldots
\]
On the other hand, the upper bound given by Serre's inequality is
\[
\ds \frac{(1+t)^4}{1-16t^2-30t^3-20t^4-5t^5} = 1+4t+22t^2+98t^3 + 493t^4 + 2313t^5 + 11288t^6 + \ldots
\]
Since the two series are not coefficientwise equal, $R$ is not Golod. We also checked that $R$ is not Golod using the Macaulay2 command \verb|isGolod(S/I)|
which computes the generators of all the Koszul homology modules, and determines whether their products are zero.
\end{remark}
\begin{example} \label{Aldo} If one is looking for an example where the ideals are generated in higher degrees, for $j \gs 1$ one can consider, along the lines of Example \ref{Sri}, the following family of products, suggested to us by Aldo Conca:
\[
\ds (x^{j+1},y^{j+1},z^{j+1},w^{j+1}) (x^j,y^j,z^j,w^j) \subseteq k[x,y,z,w].
\]
As in Example \ref{Sri}, one can show that the product of cycles
\[
\ds ((e_x \wedge e_y) \otimes x^jy^j) \cdot ((e_z \wedge e_w) \otimes z^jw^j)
\]
is not zero as a cycle and, hence, in homology. 
\end{example}

\begin{remark} \label{remSri} We want to point out that Example \ref{Sri} is not the first example of a non-Golod product of ideals that we discovered. In fact, Example \ref{Sri} was suggested to the author by Srikanth Iyengar, after some discussions about Example \ref{main}. Given the proof of Example \ref{Sri}, it becomes easy to show that the ring of Example \ref{main} is not Golod. In fact, going modulo a regular sequence of linear forms in the ring of Example \ref{main}, one obtains a ring isomorphic to the one of Example \ref{Sri}. Then, one can use \cite[Proposition 5.2.4 (2)]{AvrInf}, adapted to the graded case. The original proof that Example \ref{main} is not Golod is much more involved. Nonetheless, since this was the first example discovered by this author, we want to briefly describe the argument in the rest of this section.
\end{remark}

The original proof relies on lifting Koszul cycles. More specifically, we use the double-complex proof of the fact that $\Tor_\bullet^S(k,S/I)$ can be computed in two ways, to lift a Koszul cycle to a specific element of a finitely generated $k$-vector space. The results that we use are very well known, so we will not explain all the steps. We refer the reader to \cite{Weib} or \cite{Rot} for more details. 

Let $S=k[x_1,\ldots,x_n]$ be a polynomial ring over a field $k$, not necessarily standard graded, and let $\m$ be the irrelevant maximal ideal. Let $I \subseteq \m^2$ be a homogeneous ideal in $S$, and consider the residue class ring $R=S/I$. Since $K_\bullet$ is a free resolution of $k$ over $S$, we have that $\ds H_i(R) := H_i(K_\bullet \otimes_S R) \cong \Tor_i^S(k,R)$, and its dimension as a $k$-vector space is the $i$-th Betti number, $\beta_i$, of $R$ as an $S$-module. On the other hand, if $F_\bullet \to R \to 0$ is a minimal free resolution of $R$ over $S$, then $\ds H_i(k \otimes_S F_\bullet) \cong k \otimes F_i$ is also isomorphic to $\Tor_i^S(k,R)$. There is map $\psi: Z_i(R) \to k\otimes F_i$, which is constructed by "lifting cycles". Since the boundaries map to zero via $\psi$, this induces a map $\overline{\psi}: H_i(R) \to k\otimes_S F_i$, which is an isomorphism. See \cite{HerCycles} for a canonical way to construct Koszul cycles from elements in $k \otimes F_i$ (that is, a canonical choice of an inverse for $\psi$).

We are now ready to illustrate the example. We refer the reader to Appendix \ref{res} for an explicit expression of the differentials in a resolution of $R = S/I$ as a module over $S$.
\begin{example} \label{main} Let $k$ be a field, and let $S=k[a,b,c,d,x,y,z,w]$. Consider the monomial ideals $I_1 = (ax,by,cz,dw)$ and $I_2 = (a,b,c,d)$ inside $S$. Let 
\[
\ds I:=I_1 I_2 = (a^2x,abx,acx,adx,aby,b^2y,bcy,bdy,acz,bcz,c^2z,cdz,adw,bdw,cdw,d^2w)
\]
be their product, and set $R=S/I$. Then, the ring $R$ is not Golod.
\end{example}
\begin{proof}
Let $0 \to F_4 \to F_3 \to F_2 \to F_1 \to F_0 \to R \to 0$ be a minimal free resolution of $R$ over $S$, with maps $\varphi_j:F_j \to F_{j-1}, j=1,\ldots,4$, and $\varphi_0:F_0=S \to R$ being the natural projection. For each $i=0,\ldots,4$ and each free module $F_i = S^{\beta_i}$ fix standard bases $E^{(i)}_j$, $j=1,\ldots,\beta_i$. In this way, the differentials can be represented by matrices (see Appendix \ref{res} for an explicit description). We have the following staircase:
\[
\xymatrixcolsep{5mm}
\xymatrixrowsep{2mm}
\xymatrix{
&&&&&&&& S \otimes_S S^5 \ar[ddd]^-{1_{S} \otimes \varphi_4} \ar[rr]^-{\delta_0 \otimes 1_{S^5}} && k \otimes_S S^5 \\ \\ \\
&&&&&& K_1 \otimes_S S^{20} \ar[ddd]^-{1_{K_1} \otimes \varphi_3} \ar[rr]^-{\delta_1 \otimes 1_{S^{20}}} && S \otimes_S S^{20} \\ \\ \\
&&&&  K_{2} \otimes_S S^{30} \ar[ddd]^-{1_{K_{2}}\otimes \varphi_2} \ar[rr]^-{\delta_{2} \otimes 1_{S^{30}}} &&  K_1 \otimes_S S^{30}\\ \\ \\
&& K_{3} \otimes_S S^{16} \ar[ddd]^-{1_{K_{3}} \otimes \varphi_1} \ar[rr]^-{\delta_{3} \otimes 1_{S^{16}}} &&  K_2 \otimes_S S^{16}\\ \\ \\
K_4 \otimes_S S\ar[rr]^-{\delta_4 \otimes 1_S}  \ar[ddd]^-{1_{K_4} \otimes \varphi_0} &&  K_{3} \otimes_S S \\ \\ \\
K_4 \otimes_S R
}
\]
Let $u = (e_x \wedge e_y) \otimes (ab)$, and $v = (e_z \wedge e_w) \otimes (cd)$, inside $K_2(R) = K_2 \otimes_S R$. As they are cycles, we can consider their classes $\alpha = [u]$ and $\beta = [v]$ in homology. We want to construct a lifting $\psi(uv)$ of the Koszul cycle $uv = (e_x \wedge e_y \wedge e_z \wedge e_w) \otimes (abcd) \in K_4(R)$.
Given $(e_x \wedge e_y \wedge e_z \wedge e_w) \otimes (abcd) \in K_4 \otimes_S R$ we consider the lift $(e_x \wedge e_y \wedge e_z \wedge e_w)  \otimes  (abcd \ E^{(0)}_1) \in K_4 \otimes_S S$, and then apply the differential $\delta_4 \otimes 1_S$:
\[
\ds (\delta_4 \otimes 1_S)((e_x \wedge e_y \wedge e_z \wedge e_w) \otimes (abcd \ E^{(0)}_1)) = \begin{matrix} + (e_y \wedge e_z \wedge e_w) \otimes  (abcdx E^{(0)}_1)  \\
-  (e_x \wedge e_z \wedge e_w) \otimes (abcdy \ E^{(0)}_1) \\
+  (e_x\wedge e_y \wedge e_w) \otimes (abcdz \ E^{(0)}_1) \\
-  (e_x \wedge e_y \wedge e_z) \otimes (abcdw \ E^{(0)}_1)
\end{matrix}
\]
This is now a boundary, and, in fact, it is equal to
\[
\ds (1_{K_3} \otimes \varphi_1) \left(\begin{matrix} 
+ (e_y\wedge e_z \wedge e_w) \otimes (cd \ E^{(1)}_2) \\
- (e_x \wedge e_z \wedge e_w) \otimes (cd \ E^{(1)}_5) \\
+ (e_x\wedge e_y \wedge e_w) \otimes ( ab \ E^{(1)}_{12}) \\
- (e_x \wedge e_y \wedge e_z) \otimes (ab \ E^{(1)}_{15}) \end{matrix}\right)
\]
Now we apply $\delta_3 \otimes 1_{S^{16}}$ to this element:
\[
\ds (\delta_3 \otimes 1_{S^{16}}) \left(\begin{matrix}  
+ (e_y\wedge e_z \wedge e_w) \otimes (cd \ E^{(1)}_2) \\
- (e_x \wedge e_z \wedge e_w) \otimes (cd \ E^{(1)}_5) \\
+ (e_x\wedge e_y \wedge e_w) \otimes ( ab \ E^{(1)}_{12}) \\
- (e_x \wedge e_y \wedge e_z) \otimes (ab \ E^{(1)}_{15}) \end{matrix}\right) = \begin{matrix} + (e_z\wedge e_w) \otimes (cdy \ E^{(1)}_2 - cdx \ E^{(1)}_5) \\
 - (e_y \wedge e_w) \otimes (cdz \ E^{(1)}_2 - abx \ E^{(1)}_{12}) \\ 
+ (e_y \wedge e_z) \otimes (cdw \ E^{(1)}_2 - abx \ E^{(1)}_{15}) \\
+ (e_x \wedge e_w) \otimes (cdz \ E^{(1)}_5 - aby \ E^{(1)}_{12}) \\
- (e_x \wedge e_z) \otimes (cdw \ E^{(1)}_5 - aby \ E^{(1)}_{15}) \\ 
+ (e_x \wedge e_y) \otimes (abw \ E^{(1)}_{12} - abz \ E^{(1)}_{15}) 
\end{matrix}
\]
This is a boundary. Namely, it is equal to
\[
\ds (1_{K_2} \otimes \varphi_2)\left(\begin{matrix} 
\hspace{-3.4cm} - (e_z \wedge e_w) \otimes (cd \ E^{(2)}_{13}) \\
+ (e_y \wedge e_w) \otimes (dz \ E^{(2)}_3 +bx \ E^{(2)}_{17}+ bd \ E^{(2)}_{20}) \\
- (e_y \wedge e_z) \otimes (cw \ E^{(2)}_5 + bx \ E^{(2)}_{23} + bc \ E^{(2)}_{28}) \\
- (e_x \wedge e_w) \otimes (dz \ E^{(2)}_8 + ay \ E^{(2)}_{18} + ad \ E^{(2)}_{21}) \\
+(e_x \wedge e_z) \otimes (cw \ E^{(2)}_{10} + ay \ E^{(2)}_{24} + ac \ E^{(2)}_{29}) \\
\hspace{-3.4cm}-(e_x\wedge e_y) \otimes (ab \  E^{(2)}_{30})\end{matrix} \right)
\]
We now apply the map $\delta_2 \otimes 1_{S^{30}}$ to such a lift:
\[
\ds (\delta_2 \otimes 1_{S^{30}}) \left(\begin{matrix}
\hspace{-3.4cm} - (e_z \wedge e_w) \otimes (cd \ E^{(2)}_{13}) \\
+ (e_y \wedge e_w) \otimes (dz \ E^{(2)}_3 +bx \ E^{(2)}_{17}+ bd \ E^{(2)}_{20}) \\
- (e_y \wedge e_z) \otimes (cw \ E^{(2)}_5 + bx \ E^{(2)}_{23} + bc \ E^{(2)}_{28}) \\
- (e_x \wedge e_w) \otimes (dz \ E^{(2)}_8 + ay \ E^{(2)}_{18} + ad \ E^{(2)}_{21}) \\
+(e_x \wedge e_z) \otimes (cw \ E^{(2)}_{10} + ay \ E^{(2)}_{24} + ac \ E^{(2)}_{29}) \\
\hspace{-3.4cm} -(e_x\wedge e_y) \otimes (ab \  E^{(2)}_{30})\end{matrix} \right) = 
\]
\vspace{0.5cm}
\[
=  \begin{matrix} + e_x \otimes (dzw \ E^{2}_8 - czw \ E^{(2)}_{10} + ayw \  E^{(2)}_{18} + adw \ E^{(2)}_{21} - ayz \ E^{(2)}_{24} - acz \ E^{(2)}_{29} + aby \ E^{(2)}_{30}) \\
-e_y \otimes (dzw \  E^{(2)}_3 - czw \  E^{(2)}_5 + bxw \ E^{(2)}_{17} + bdw \ E^{(2)}_{20} - bxz \ E^{(2)}_{23} - bcz \ E^{(2)}_{28}+abx \  E^{(2)}_{30}) \\
+e_z \otimes (-cyw \ E^{(2)}_5 + cxw \ E^{(2)}_{10} + cdw \ E^{(2)}_{13} - bxy \ E^{(2)}_{23} + axy \ E^{(2)}_{24} - bcy \ E^{(2)}_{28} + acx \ E^{(2)}_{29}) \\
- e_w \wedge (-dyz \ E^{(2)}_3 + dxz \ E^{(2)}_8 + cdz \ E^{(2)}_{13} - bxy \ E^{(2)}_{17} + axy \ E^{(2)}_{18} - bdy \ E^{(2)}_{20}+adx \ E^{(2)}_{21})\end{matrix}
\]
Again, this element is a boundary. In fact, it is equal to
\[
\ds (1_{K_1} \otimes \varphi_3) \left(\begin{matrix} + e_x \otimes (zw \ E^{(3)}_7 + a \ E^{(3)}_{20}) \\
-e_y \otimes (zw \ E^{(3)}_4 + b \ E^{(3)}_{19}) \\
+ e_z \otimes (xy \ E^{(3)}_{13} + c \ E^{(3)}_{18}) \\
- e_w \otimes (xy \ E^{(3)}_{10} + d \ E^{(3)}_{17}) \end{matrix}\right)
\]
One more time, we apply $\delta_1 \otimes 1_{S^{20}}$, to get
\[
(\delta_1 \otimes 1_{S^{20}}) \left(\begin{matrix} + e_x \otimes (zw \ E^{(3)}_7 + a \ E^{(3)}_{20}) \\
-e_y \otimes (zw \ E^{(3)}_4 + b \ E^{(3)}_{19}) \\
+ e_z \otimes (xy \ E^{(3)}_{13} + c \ E^{(3)}_{18}) \\
- e_w \otimes (xy \ E^{(3)}_{10} + d \ E^{(3)}_{17}) \end{matrix}\right) = 
\]

\[
= \begin{matrix} 1 \otimes (-yzw \ E^{(3)}_4+xzw \ E^{(3)}_7-  xyw \ E^{(3)}_{10}+ xyz \  E^{(3)}_{13}-dw \ E^{(3)}_{17}+cz \ E^{(3)}_{18}-by \  E^{(3)}_{19} + ax \ E^{(3)}_{20}). \end{matrix}
\]
This is a boundary: it is equal to $(1_S \otimes \varphi_4)(1 \otimes E^{(4)}_5)$. When applying $\delta_0 \otimes 1_{S^5}$ to the lift, we finally get the image of $uv$ under the map $\psi:Z_4(R) \to k \otimes_S S^5$. Namely:
\[
\psi(uv) = (\delta_0 \otimes 1_{S^5})(1 \otimes E^{(4)}_5) = \overline{1} \otimes E^{(4)}_5 \in k \otimes_S S^5
\]
and since the latter is non-zero, because it is part of a $k$-basis of $k \otimes_S S^5$, we obtain that $uv$ is not a boundary of the Koszul complex. Thus, $\alpha\beta$ is non-zero in $H_4(R)$, and $R$ is not Golod.
\end{proof}
\begin{remark} As pointed out in Remark \ref{remSri}, with the same notation as in Example \ref{main}, we have that $x-a,y-b,z-c,w-d$ is a regular sequence modulo $I$. Going modulo such a regular sequence of linear forms, one recovers the ring of Example \ref{Sri}.
\end{remark}

Recall that a monomial ideal $I$ satisfies the strong-GCD condition (see \cite[Definition 3.8]{Joll}) if there exists a linear order $\prec$ on the set $\MinGen(I)$ of minimal monomial generators of $I$ such that, for any two monomials $u \prec v$ in $\MinGen(I)$, with $\gcd(u,v) = 1$, there exists a monomial $w \in \MinGen(I)$, $v \ne w$, with $u \prec w$ and such that $w$ divides $uv$. The ring in Example \ref{main} satisfies the strong-GCD condition, being a product (see \cite[Theorem 1.1]{FkhWel} and the discussion in the Introduction). We present here another ideal that satisfies the strong-GCD condition, that is not Golod. Although it is not a product, it has the advantage of having fewer generators than our previous ideals. Another example has been discovered by Lukas Katth{\"a}n  \cite{LK}, who considers the ideal $I=(x_1x_2y,x_2x_3y,x_3x_4y,x_4x_5,x_5x_1)$ in the polynomial ring $k[x_1,x_2,x_3,x_4,x_5,y]$.
\begin{example} \label{exstrgcd} Let $S = k[x,y,z]$, and let $I=(x^2y,xy^2,x^2z,y^2z,z^2)$. Set $R=S/I$. The ideal $I$ satisfies the strong-GCD condition, for example choosing $x^2y \prec xy^2 \prec x^2z \prec y^2z \prec z^2$. Using Macaulay2 \cite{Mac2}, we checked that the Poincar{\'e} series of $R$ starts as
\[
\ds P_R(t) = 1+3t+8t^2 + 21t^3+55t^4+144t^5+377t^6+ \ldots
\]
and that the right-hand side of Serre's inequality is
\[
\ds \frac{(1+t)^3}{1-5t^2-5t^3-t^4} = 1+3t+8t^2 + 21t^3+56t^4+148t^5+393t^6+ \ldots
\]
Therefore, $R$ is not Golod. Alternatively, one can use the Macaulay2 command \verb|isGolod(S/I)|, or one can show, with arguments similar to the ones used above, that the product of Koszuyl cycles
\[
\ds \left((e_x \wedge e_y) \otimes xy \right) \cdot \left(e_z \otimes z\right) \in K_3(R)
\]
is not zero in homology. Looking for a squarefree example, using polarization, one obtains that $I'=(axy,bxy,axz,byz,cz) \subseteq k[a,b,c,x,y,z]$ satisfies the strong GCD condition, and is not Golod.
\end{example}
\section{Strongly Golod property for rational powers of monomial ideals} \label{frac}
Let $k$ be a field, and let $S = k[x_1,\ldots,x_n]$, with $\deg(x_i) =  d_i >0$. We recall the definition of rational powers of an ideal.
\begin{definition} For an ideal $I \subseteq S$ and positive integers $p,q$ define the ideal
\[
\ds I^{p/q} := \{f \in R \mid f^q \in \ov{I^p}\}.
\]
\end{definition}
The integral closure of $I^p$ inside the definition is needed in order to make the set into an ideal, and to make it independent of the choice of the representation of $p/q$ as a rational number.
\begin{remark} We would like to warn the reader about a potential source of confusion. When $p=q$, the ideal $I^{p/q} = I^{1/1}$ is the integral closure $\overline{I}$ of $I$, and should not be regarded as the ideal $I^1 = I$, even though the exponents $1/1$ and $1$ are equal.
\end{remark}
\begin{remark} If $I \subseteq S$ is a monomial ideal, then so is $I^{p/q}$.
\end{remark}
\begin{proof} Let $f=\sum_{i=1}^d \lambda_i u_i \in I^{p/q}$, where $0 \ne \lambda_i \in k$ and $u_i$ are monomials. Since $\ov{I^p}$ is monomial, and we have that $f^{qr} \in I^{pr}$ for all integers $r \gg 0$ \cite[Theorem 1.4.2]{Herzog_Hibi}. Also, $I^{pr}$ is monomial, therefore every monomial appearing in $f^{qr}$ belongs to $I^{pr}$, and in particular for any $i=1,\ldots,d$ we have that $u_i^{qr} \in I^{pr}$ for all $r \gg 0$. This shows that $u_i^q \in \ov{I^{p}}$, that is $u_i \in I^{p/q}$ for all $i=1,\ldots,d$, and hence $I^{p/q}$ is monomial.
\end{proof}

In the rest of the section, we assume that the characteristic of $k$ is zero.
\begin{definition}[\cite{HerHun}] A proper homogeneous ideal $I \subseteq S$ is called {\it strongly Golod} if $\partial(I)^2 \subseteq I$. 
\end{definition} 
Here, $\partial(I)$ denotes the ideal of $S$ generated by the partial derivatives of elements in $I$. By \cite[Theorem 1.1]{HerHun}, if $I$ is strongly Golod, then $S/I$ is Golod. This condition, however, is only sufficient. For example, the ideal $I=(xy,xz) \subseteq k[x,y,z]$ is Golod \cite{Shamash}, or \cite[Proposition 5.2.5]{AvrInf}. However, it is not strongly Golod. This example is not even squarefree strongly Golod (see Section \ref{lcm} for the definition). In case $I$ is monomial, being strongly Golod is equivalent to the requirement that, for all minimal monomial generators $u,v \in I$, and all integers $i,j$ such that $x_i$ divides $u$ and $x_j$ divides $v$, one has $uv/x_ix_j \in I$. 

The following argument is a modification of \cite[Proposition 3.1]{HerHun}. 
\begin{theorem} \label{fractional_monomial} Let $I \subseteq S$ be a strongly Golod monomial ideal. If $p \gs q$, then $I^{p/q}$ is strongly Golod.
\end{theorem}
\begin{proof} Let $u \in I^{p/q}$ be a monomial generator, then $u^{qr} \in I^{pr}$ for all $r \gg 0$. Let $j$ be an index such that $x_j \mid u$, we claim that $(u/x_j)^{qr} \in I^{pr/2}$ for all even $r\gg 0$. Notice that if $x_j^2 \mid u$ then, for any even $r$, $r \gg 0$, we have
\[
\ds \left(\frac{u}{x_j}\right)^{qr} = u^{q(r/2)} \left(\frac{u}{x_j^2}\right)^{qr/2} \in I^{p(r/2)},
\]
as desired. Now suppose that $x_j$ divides $u$, but $x_j^2$ does not. Since for any $r \gg 0$ we have that $u^{qr} \in I^{pr}$, we can write
\[
u^{qr} = m_1 m_2 \cdots m_{pr},
\]
where $m_i \in I$ for all $i$. Again, we can assume that $r$ is even. For $i =1,\ldots,pr$ let $d_i$ be the maximum non-negative integer such that $x_j^{d_i}$ divides $m_i$. Then we can rewrite
\[
u^{qr} = m_1 \cdots m_a m_{a+1} \cdots m_{a+b} m_{a+b+1} \cdots m_{pr}, 
\]
where $d_i=0$ for $1\ls i \ls a$, $d_i=1$ for $a+1 \ls i \ls a+b$ and $d_i \gs 2$ for $a+b+1 \ls i \ls pr$. Because of the assumption $x_j^2 \not| \ u$ we have that 
\[
\ds qr = \sum_{i=1}^{pr} d_i = b+\sum_{i=a+b+1}^{pr} d_i \gs b + 2(pr-b-a).
\]
But we assumed that $p \gs q$, therefore $pr \gs b + 2(pr-b-a)$, which gives $a+b/2 \gs pr/2$ and also $a+\lfloor b/2 \rfloor \gs pr/2$ because $pr/2$ is an integer. Write
\[
\ds \left(\frac{u}{x_j}\right)^{qr} =  \frac{u^{qr}}{x_j^{qr}} = m_1 \cdots m_a \frac{m_{a+1}}{x_j} \cdots \frac{m_{a+b}}{x_j} \frac{m_{a+b+1} \cdots m_{pr}}{x_j^{qr-b}},
\]
then $m_{a+1}\ldots m_{a+b}/x_j^b \in I^{\lfloor b/2 \rfloor}$ because $I$ is strongly Golod, so that $\partial(I)^b \subseteq I^{\lfloor \frac{b}{2} \rfloor}$. Furthermore, $m_1\cdots m_a \in I^a$. Therefore
\[
\ds \left(\frac{u}{x_j} \right)^{qr} \in I^{a+\lfloor b/2 \rfloor} \subseteq I^{pr/2}.
\]
Now let $v \in I^{p/q}$ be another monomial generator, and assume that $x_i | v$. Then, for all even $r \gg 0$ we have
\[
\ds \left(\frac{uv}{x_jx_i} \right)^{qr} \in I^{pr}
\]
which implies that $uv/x_jx_i \in I^{p/q}$. Since $u$ and $v$ were arbitrary monomial generators, $I^{p/q}$ is strongly Golod.
\end{proof}
\begin{corollary}\cite[Proposition 3.1]{HerHun} \label{int_cl} Let $I \subseteq S$ be a monomial strongly Golod ideal, then $\ov I$ is strongly Golod.
\end{corollary}
\begin{proof} Choose $p=q$ in Theorem \ref{fractional_monomial}.
\end{proof}
\begin{proposition} \label{fractional_monomial_gen} Let $I \subseteq S$ be a monomial ideal. If $p \gs 2q$, then $I^{p/q}$ is strongly Golod.
\end{proposition}
\begin{proof} Since the ideal $I^{p/q}$ does not depend on the representation of $p/q$ as a rational number, without loss of generality we can assume that $p$ is even. Let $u \in I^{p/q}$ be a monomial generator, so that $u^{qr} \in I^{pr}$ for all $r \gg 0$. Let $j$ be such that $x_j \mid u$, then we can write
\[
\ds \left(\frac{u}{x_j}\right)^{qr} = \frac{u^{qr}}{x_j^{qr}} = \frac{m_1 \ldots m_{pr}}{x_j^{qr}} = m_1 \ldots m_{pr-qr} \cdot \frac{m_{pr-qr+1} \ldots m_{pr}}{x_j^{qr}},
\]
for some $m_i \in I$. But then $(u/x_j)^{qr} \in I^{pr-qr} \subseteq I^{pr/2}$ because $q \ls p/2$. Let $v \in I^{p/q}$ be another monomial generator and assume that $x_i \mid v$. For $r\gg 0$ we have again that $(v/x_i)^{qr} \in I^{pr/2}$, so that
\[
\ds \left( \frac{uv}{x_jx_i}\right)^{qr} \in I^{pr}
\]
for all $r \gg 0$, and thus $uv/x_jx_i \in I^{p/q}$. Since $u$ and $v$ were arbitrary, we have that $I^{p/q}$ is strongly Golod.
\end{proof}
If $I$ is not strongly Golod and $2q > p \gs q$, it is not true in general that $I^{p/q}$ is strongly Golod, as the following family of examples shows.
\begin{example} \label{exs} Let $2q > p \gs q$ be two positive integers and consider the ideal $I=(xy,z^q)$ inside the polynomial ring $S=k[x,y,z]$, where $k$ is a field of characteristic zero. Then
\[
(xy)^q (z^q)^{p-q} = (xyz^{p-q})^q \in I^p,
\]
that is $xyz^{p-q} \in I^{p/q}$. Thus, $u:=y^2z^{2p-2q} \in \partial(I^{p/q})^2$. On the other hand, $u^{q} =y^{2q}z^{2pq-2q^2} \notin \ov{I^{p}}$ because the only monomial generator of $I^p$ that can appear in an integral relation for $u$ is $z^pq$. But 
\[
\ds y^{2qn}z^{2pqn-2q^2n} \notin (z^{pqn})
\]
for any $n$ because $pqn > 2pqn-2q^2n \Longleftrightarrow p<2q$, and we have the latter by assumption. As a consequence, $u \notin I^{p/q}$, and thus $I^{p/q}$ is not strongly Golod.
\end{example}
\begin{remark}
If we choose $p=q=2$ in Example \ref{exs}, we have in addition that $I^{p/q} = \ov I = I = (xy,z^2)$ is not even Golod, because it is a complete intersection of height two.
\end{remark}

As a consequence, not all integrally closed ideals, even if assumed monomial, are Golod. A more trivial example is the irrelevant maximal ideal $\m$. However, as noted above in Corollary \ref{int_cl}, if $I$ is a strongly Golod monomial ideal, then $\ov I$ is strongly Golod. More generally, if $I \subseteq S=k[x_1,\ldots,x_n]$ is homogeneous, then $\ov{I^j}$ is strongly Golod for all $j \gs n+1$ \cite[Theorem 2.11]{HerHun}. It is still an open question whether $\ov{I^j}$ is strongly Golod, or, at least, Golod, for any ideal $I$ and $j \gs 2$. Since for $j \gs 2$, the ideal $I^j$ is strongly Golod, one can ask the following more general question, which has already been raised by Craig Huneke:
\begin{question}\cite[Problem 6.19]{McCPeeva} \label{quest_int} Let $I \subseteq S$ be a strongly Golod ideal. Is $\ov{I}$ [strongly] Golod?
\end{question}
\begin{remark} We checked with Macaulay2 \cite{Mac2} that the ideal $I$ of Example \ref{main} is integrally closed. Therefore, the integral closure of a product of ideals, even monomial ideals in a polynomial ring, may not be Golod.
\end{remark}
We end the section with a more generic question about Golodness of the ideal $I^{3/2}$. Note that for each ideal $I = (xy,z^q)$ of the family considered in Example \ref{exs}, the rational power $I^{3/2}$ is not strongly Golod. However, it is Golod. In fact, it is not hard to see that $\ov{I^3} = I^3 = (x^3y^3,x^2y^2z^q,xyz^{2q},z^{3q})$. As a consequence, we have $I^{3/2} = (x^2y^2,xyz^{\lceil\frac{q}{2}\rceil},z^{\lceil\frac{3q}{2}\rceil})$. Consider the linear form $x-y \in \m \smallsetminus \m^2$, which is a non-zero divisor modulo $I^{3/2}$. The image of $I^{3/2}$ in the polynomial ring $S' = S/(x-y) \cong k[x,z]$ is $(x^4,x^2z^{\lceil\frac{q}{2}\rceil},z^{\lceil\frac{3q}{2}\rceil})$. Such an ideal is easily seen to be strongly Golod, hence Golod. By \cite[Proposition 5.2.4 (2)]{AvrInf}, the ideal $I^{3/2}$ is then Golod. 
\begin{question} Let $I \subseteq S$ be a proper homogeneous ideal. Is $I^{3/2}$ always Golod? Is it true if $I$ is monomial?
\end{question}

\section{$\lcm$-strongly Golod monomial ideals} \label{lcm}
Let $k$ be a field, and let $S=k[x_1,\ldots,x_n]$, with $\deg(x_i) = d_i > 0$. 

\begin{definition}
Let $m \in S$ be a monomial, and let $I \subseteq S$ be a monomial ideal. Define $I_{m} \subseteq I$ to be the ideal of $S$ generated by the monomials of $I$ which divide $m$. We say that $I$ is $m$-divisible if $I = I_m$.
\end{definition}
\begin{remark}\label{sqfstrg} Note that, choosing $m=x_1\cdots x_n$, then $m$-divisible simply means squarefree. 
\end{remark}

We now recall the Taylor resolution of a monomial ideal. Let  $I \subseteq S$ be a monomial ideal, with minimal monomial generating set $\{m_1,\ldots,m_t\}$. For each subset $\Lambda \subseteq [t] :=\{1,\ldots,t\}$ let $L_\Lambda:= \lcm(m_i \mid i \in \Lambda)$. Let $a_\Lambda \in \NN^n$ be the exponent vector of the monomial $L_\Lambda$, and let $S(-a_\Lambda)$ be the free module, with generator in multi-degree $a_\Lambda$. Consider the free modules $T_i:= \bigoplus_{|\Lambda| = i} S(-a_\Lambda)$, with basis $\{e_\Lambda \}_{|\Lambda| = i}$. Also, set $F_0:=S$. The differential $\tau_i:T_{i} \to T_{i-1}$ acts on an element of the basis $e_{\Lambda}$, for $\Lambda \subseteq [t]$, $|\Lambda|=i$, as follows:
\[
\ds \tau_i(e_\Lambda) = \sum_{j \in \Lambda} {\rm sign}(j,\Lambda) \cdot \frac{L_{\Lambda}}{L_{\Lambda \smallsetminus \{j\}}}\cdot e_{\Lambda \smallsetminus \{j\}}
\]
Here ${\rm sign}(j,\Lambda)$ is $(-1)^{s+1}$ if $j$ is the $s$-th element in the ordering of $\Lambda \subseteq [t]$. The resulting complex is a free resolution of $S/I$ over $S$, called the Taylor resolution. The following was already noted in \cite[Corollary 3.2]{BrHe_Multigraded}, and \cite[Corollary to Theorem 1]{Sri_Shifts}.
\begin{remark} \cite[Corollary 3.2]{BrHe_Multigraded} \label{res_m_graded} Let $I$ be an $m$-divisible monomial ideal. Then, the Koszul homology $H_\bullet(S/I)$ is $\ZZ^n$-multigraded, and it is concentrated in multidegrees $a_{\Lambda'} \in \NN^n$ such that the monomial $x_1^{(a_{\Lambda'})_1}\cdots x_n^{(a_\Lambda')_n} = L_{\Lambda'}$ divides $m$.
\end{remark}
In \cite{HerHun}, given a squarefree monomial ideal, Herzog and Huneke introduce the notion of squarefree strongly Golod monomial ideal. Given Remark \ref{sqfstrg}, we generalize it to the notion of $\lcm$-strongly Golod. Let $I$ be a monomial ideal, and let $m:=\lcm(I)$ be the least common multiple of the monomials appearing in the minimal monomial generating set of $I$. By definition, $I$ is always $m$-divisible. Also, if $I$ is $m'$-divisible for some other monomial $m'$, then $m$ divides $m'$.

In what follows, we assume that the characteristic of $k$ is zero.
\begin{definition}
Let $I \subseteq S$ be a monomial ideal, and let $m:=\lcm(I)$ be as defined above. Let $\partial (I)^{[2]}$ denote the ideal $(\partial (I)^2)_{m}$. We say that $I$ is $\lcm$-strongly Golod if $\partial(I)^{[2]} \subseteq I$.
\end{definition}
We also make the following definition
\begin{definition} Let $I \subseteq S$ be a proper homogeneous ideal. We say that $R=S/I$ is weakly Golod if the multiplication on Koszul homology is identically zero.
\end{definition}

The following is the main result of the section, and justifies the previous definition. It is a generalization of \cite[Theorem 3.5]{HerHun}.
\begin{theorem} \label{lcm_main}
Let $I \subseteq S$ be an $\lcm$-strongly Golod monomial ideal. Then, $S/I$ is weakly Golod.
\end{theorem}
\begin{proof}
Let $m := \lcm(I)$, so that $I$ is $m$-divisible and $\partial(I)^{[2]} \subseteq I$. By Remark \ref{res_m_graded}, we can choose a $k$-basis of $H_\bullet(S/I)$ consisting of elements of multidegrees $\alpha_\Lambda$, where $\underline{x}^{\alpha_{\Lambda}}$ divides $m$. Let $a,b$ be two such elements. If $ab$ has multidegree $\alpha \in \NN^n$, such that $\underline{x}^{\alpha}$ does not divide $m$, then necessarily $ab=0$ because of the multigrading on $H_\bullet(S/I)$. So assume that the multidegree $\alpha$ of $ab$ is such that $\underline{x}^{\alpha}$ divides $m$. By \cite{HerCycles}, $a$ and $b$ can be represented by cycles whose coefficients are $k$-linear combinations of elements in $\partial(I)$. Since $I$ is monomial, so is $\partial(I)$. Because of the multidegree of $ab$, we then have that $a$ and $b$ can be represented by cycles whose coefficients are $k$-linear combinations of monomials $u,v \in \partial(I)$, such that the products $uv$ divide $m$. Then $uv \in \partial(I)^{[2]} \subseteq I$ for each product $uv$ appearing in these sums, and, as a consequence, $ab = 0$ in $H_\bullet(S/I)$.
\end{proof}
\begin{disc} It is easy to see that being $\lcm$-strongly Golod is only sufficient to be weakly Golod. For example, the ideal $(xy,xz) \subseteq k[x,y,z]$ is even Golod \cite{Shamash}, but not $\lcm$-strongly Golod. The proof of Theorem \ref{lcm_main}, as well as the proofs of \cite[Theorem 1.1]{HerHun} and \cite[Theorem 3.5]{HerHun}, are based on a canonical description of Koszul cycles whose residue classes form a $k$-basis for the Koszul homology $H_\bullet(S/I)$ \cite{HerCycles}. We want to suggest a slightly different definition of strong Golodness:

Potentially, one has to check that $\partial_i(f) \partial_j(g) \in I$ for any $f,g \in I$, and any $i,j = 1,\ldots, n$, where $\partial_i = \partial/\partial x_i$ and $\partial_j = \partial/\partial x_j$. However, by \cite{HerCycles}, each $\partial_i(f)$ appears as a factor in some coefficient of a Koszul cycle, which has the form $\ds (e_i \wedge \ldots) \otimes \partial_i(f) \in K_\bullet \otimes S/I = K_\bullet(S/I)$. Therefore, the corresponding product $\partial_i(f) \partial_j(g)$ will appear inside some coefficient of the form
\[
\ds (e_i \wedge e_j \wedge \ldots) \otimes \partial_i(f) \partial_j(g).
\]
For $i=j$, we have that $e_i \wedge e_j = 0$. Hence we may consider only products $\partial_i(f) \partial_j(g)$, for $i\ne j$, in the definition of strongly Golod and $\lcm$-strongly Golod. With this modification, the ideal $(xy,xz)$ becomes $\lcm$-strongly Golod. The ideal $(x^2,xy)$ in the polynomial ring $k[x,y]$, which is $\lcm$-strongly Golod, with this modification becomes strongly Golod. In fact, $\frac{xy}{x} \cdot \frac{xy}{x} = y^2 \notin (x^2,xy)$ is the product that is preventing it from being strongly Golod. However, the partial derivatives, in this case, are both with respect to $x$, so we can disregard such a product.
\end{disc}
Here follows an example of a non-squarefree ideal which is $\lcm$-strongly Golod, but not strongly Golod, even with the modified definition.
\begin{example} Let $k$ be a field of characteristic zero, and let $I = (x^2y^2,x^2z,y^2z) \subseteq k[x,y,z]$. Then $I$ is not strongly Golod, even in the definition suggested above. In fact, $xz,yz \in \partial(I)$ come from taking derivative with respect to $x$ and $y$, respectively, but their product is $xz \cdot yz \notin I$. However, such an element does not divide $\lcm(I) = x^2y^2z$, therefore it can be disregarded when looking at the $\lcm$-strongly Golod condition. In fact, one can check that $\partial(I)^{[2]} \subseteq I$, that is, $I$ is $\lcm$-strongly Golod in this case.
\end{example}
As shown in \cite[Proposition 3.7]{HerHun} for the squarefree part, if $m$ is a monomial in $S$, then the $m$-divisible part of a strongly Golod monomial ideal is $\lcm$-strongly Golod. We record it in the next proposition.
\begin{proposition}
Let $I \subseteq S$ be a strongly Golod monomial ideal, and let $m$ be a monomial. Then $I_m$ is $\lcm$-strongly Golod. In particular, $I$ is $\lcm$-strongly Golod.
\end{proposition}
\begin{proof}
We have that 
\[
\ds \partial (I_m)^{[2]} = (\partial (I_m)^2)_m \subseteq (\partial(I)^2)_m \subseteq I_m.
\]
\end{proof}

As mentioned in Section \ref{frac}, if $I$ is a strongly Golod monomial ideal, then $\ov I$ is strongly Golod. It is natural to ask the following question:
\begin{question}  If $I \subseteq S$ is an $\lcm$-strongly Golod monomial ideal, is $\ov I$ ($\lcm$-strongly) Golod? For integers $p \gs q$, is the ideal $I^{p/q}$ ($\lcm$-strongly) Golod?
\end{question}
The inequality $p \gs q$ seems reasonable to require, given previous results.
\section{Golodness of products and further questions} \label{quests}
Throughout this section, unless otherwise specified, $k$ is a field of characteristic zero, and $S=k[x_1,\ldots,x_n]$ is a polynomial ring over $k$, with $\deg(x_i) = d_i > 0$. It is easy to see that arbitrary intersections of strongly Golod ideals are strongly Golod \cite[Theorem 2.3 (a)]{HerHun}. Given a proper homogeneous ideal $I \subseteq S$, one may ask what is the intersection of all the strongly Golod ideals containing $I$. In other words, what is the smallest ideal that contains $I$ and that is strongly Golod. Clearly, such an ideal must contain $I+\partial(I)^2$. On the other hand, note that $\partial(\partial(I)^2)) \subseteq \partial(I)$, therefore
\[
\ds \partial(I+\partial(I)^2)^2 \subseteq I+\partial(I)^2.
\]
Thus, $I+\partial(I)^2$ is strongly Golod, and it is indeed the smallest strongly Golod ideal containing $I$. 

We now introduce a sufficient condition, which is far from being necessary, for the product of two ideals to be strongly Golod.
\begin{definition} \label{defn_str_Gol_pairs} Let $S=k[x_1,\ldots,x_n]$ and let $I,J \subseteq S$ be two ideals. $(I,J)$ is called a strongly Golod pair if $\partial(I)^2 \subseteq I:J$ and $\partial(J)^2 \subseteq J:I$.
\end{definition}
Note that, for examples of small size, the conditions from Definition \ref{defn_str_Gol_pairs} can easily be checked with the aid of a computer. The following proposition is the main motivation behind the definition.
\begin{proposition} \label{str_Gol_pairs} If $(I,J)$ is a strongly Golod pair, then $IJ$ is strongly Golod.
\end{proposition}
\begin{proof} We noted above that the smallest strongly Golod ideal containing $IJ$ is $IJ + \partial(IJ)^2$. In our assumptions, we have
\[
\ds \partial(IJ)^2 \subseteq (\partial(I)J + I \partial(J))^2 \subseteq \partial(I)^2J^2 + I^2 \partial(J)^2 + IJ \subseteq IJ.
\]
Therefore, $IJ + \partial(IJ)^2 = IJ$, which is then strongly Golod.
\end{proof}
Note that, looking at the proof of Proposition \ref{str_Gol_pairs}, one may notice that the conditions $\partial(I)^2 \subseteq IJ:J^2$ and $\partial(J)^2 \subseteq IJ:I^2$ are sufficient in order for the product $IJ$ to be strongly Golod. However, when studying properties of the product $IJ$, one can replace the ideal $I$ with $IJ:J$ without affecting the product. In fact:
\[
\ds IJ \subseteq (IJ:J)J \subseteq IJ,
\]
forcing equality. Repeating the process, one gets an ascending chain of ideals containing $I$, that eventually stabilizes. Therefore one can assume that $IJ:J = I$. Similarly, one can assume that $IJ:I = J$. Therefore the conditions above become
\[
\ds \partial(I)^2 \subseteq IJ:J^2 = (IJ:J):J = I:J,
\]
which is precisely the requirement in the definition of strongly Golod pair. Similarly for the other colon ideal. Of course, as long as one can write an ideal in terms of a Golod pair, one gets that the ideal is strongly Golod. Therefore, one may keep in mind the weaker colon conditions that come from the proof of Proposition \ref{str_Gol_pairs}.
Examples of strongly Golod pairs include:
\begin{enumerate}[(1)]
\item $(I^r,I^s)$, for any proper ideal $I \subseteq S$ and any integers $r,s \gs 1$.
\item If $I$ and $J$ are strongly Golod, then $(I,J)$ is a strongly Golod pair.
\item If $I \subseteq J$ and $I$ is strongly Golod, then $(I,J)$ is a strongly Golod pair.
\item $(I,I:\partial(I)^2)$ is a strongly Golod pair for any proper ideal $I \subseteq S$.
\end{enumerate}

\begin{remark} Let $I_1,\ldots,I_n$ be proper ideals in $S$. Assume that, for all $i = 1,\ldots,n$ there exists $j \ne i$ such that $(I_i,I_j)$ is a strongly Golod pair, then the product $I:=I_1\cdots I_n$ is strongly Golod. In fact
\[
\partial(I_1I_2\ldots I_n) \subseteq \partial(I_1)I_2\ldots I_n + I_1 \partial(I_2)\ldots I_n+ \ldots + I_1I_2\ldots\partial(I_n).
\]
Thus
\[
\partial(I)^2 \subseteq \partial(I_1)^2I_2^2\cdots I_n^2 + I_1^2 \partial(I_2)^2\cdots I_n^2+ \ldots + I_1^2I_2^2\cdots\partial(I_n)^2 + I.
\]
By assumption, for each $i$ there exists $j\ne i$ such that $\partial(I_i)^2I_j \subseteq I_i$, and the claim follows. More generally, one could define $(I_1,\ldots,I_n)$ to be a strongly Golod $n$-uple provided 
\[
\ds \partial(I_i)^2 \subseteq I: (I_1 \cdots I_{i-1}\cdot I_{i+1} \cdots I_n)
\]
for all $i =1,\ldots,n$. Then, the above argument shows that if $(I_1,\ldots,I_n)$ is a strongly Golod $n$-uple, the product $I_1 \cdots I_n$ is strongly Golod.
\end{remark}

All the conditions discussed above are sufficient, but evidently not necessary, for a product of two ideals to be Golod. We raise the following general question:
\begin{question} Is there some relevant class of [pairs of] ideals for which products are [strongly] Golod?
\end{question}
In particular, note that in all the examples of Section \ref{Sec_counter}, the ideals appearing in the product are not Golod. It is then natural to ask:
\begin{question} If one of the two ideals $I_1,I_2$ [or both] is Golod, is then $S/I_1I_2$ Golod?
\end{question}
Another problem relating Golod rings to products is the following. Let $I,J$ be two proper homogeneous ideals in a polynomial rings $S = k[x_1,\ldots,x_n]$, with $\m=(x_1,\ldots,x_n)$. Suppose that $S/IJ$ is Cohen-Macaulay. In \cite{HunProd}, Huneke asks whether the Cohen-Macaulay type, that is, $t(S/IJ) = \dim_k \Ext^{\Depth(S/IJ)}(k,S/IJ)$, is always at least the height of $IJ$. This was motivated by the fact that Gorenstein rings are never products, unless they are hypersurfaces. Thus, when $S/IJ$ is Cohen-Macaulay and not a hypersurface, the type is always at least two. As noted in \cite{HunProd}, the case when $I=\m$ and $J$ is $\m$-primary, follows by Krull's height theorem. In our context, it seems natural to ask the following question:
\begin{question}
Let $I \subseteq S$ be a homogeneous ideal such that $S/I$ is Cohen-Macaulay and Golod. Is it true that the Cohen-Macaulay type $t(S/I)$ is always at least $\Ht(I)$? Is the Cohen-Macaulay assumption needed?
\end{question}
As pointed out by Frank Moore, there is still no example of a ring which is not Golod, and for which the second Massey operation is trivial. With our notation, this means that it is not known whether weakly Golod rings are always Golod. As explained earlier in this article, it has been claimed in \cite[Theorem 5.1]{BerJoll} that this is true in case the ideal is monomial. Given the examples in Section \ref{Sec_counter}, and our previous discussions, we do not know whether such an argument is still valid. We decided to raise the general question here again:
\begin{question} Are weakly Golod rings always Golod? Is it true for monomial ideals? Is there any relevant class of ideals for which this holds true?
\end{question}
In \cite[Proposition 2.12]{HerHun}, Herzog and Huneke show that the Ratliff Rush filtration of a strongly Golod ideal is strongly Golod. We obtain a similar statement, using Golod pairs.
\begin{proposition} \label{Ratliff} If $(I,J)$ is a strongly Golod pair, then the ideal
\[
\ds \bigcup_{n \gs 0} \left(I^{n+1}J : I^n\right)
\]
is strongly Golod.
\end{proposition}
\begin{proof}
Let $f \in S$ be such that $f I^{n-1} \subseteq I^{n}J$ for some $n$. Then $fI^n \subseteq I^{n+1}J$. Let $\partial$ denote a partial derivative with respect to any variable. Taking partial derivatives, from the containment above we obtain that
\[
\partial(f)I^n \subseteq fI^{n-1} \partial(I) + I^n \partial(I)J + I^{n+1} \partial(J) \subseteq I^n \partial(I)J + I^{n+1}\partial(J).
\]
Let $f,g \in \bigcup_{n \gs 0} I^{n+1}J:I^n$ and choose $n \gg 0$ such that $fI^{n-1} \subseteq I^{n}J$ and $gI^{n-1} \subseteq I^{n}J$. Then
\[
\partial(f)\partial(g)I^{2n} \subseteq I^{2n} \partial(I)^2 J^2 + I^{2n+2} \partial(J)^2 + I^{2n+1}J \subseteq I^{2n+1}J
\]
because $\partial(I)^2 J \subseteq I$ and $I \partial(J)^2 \subseteq J$.
\end{proof}
In particular, Proposition \ref{Ratliff} shows that the Ratliff-Rush closure of any power $I^d$, $d \gs 2$, is strongly Golod. In fact, it is enough to apply Proposition \ref{Ratliff} to the strongly Golod pair $(I^{d-1},I)$. This already follows from \cite[Proposition 2.12]{HerHun}, since $I^d$ is strongly Golod for any $d \gs 2$.

Given that the Ratliff-Rush closure of a strongly Golod ideal is strongly Golod we ask:
\begin{question} \label{coeff_id} Given a strongly Golod ideal $I \subseteq S$, is every coefficient ideal of $I$ [strongly] Golod?
\end{question}
Question \ref{coeff_id} is a more general version of Question \ref{quest_int}. In fact, both the integral closure and the Ratliff-Rush closure are coefficient ideals. See \cite{Shah} for details about coefficient ideals.

We conclude the section with two questions regarding the notion of strongly Golod ideal. The definition of strongly Golod ideals is restricted to homogeneous ideals in a polynomial ring $S=k[x_1,\ldots,x_n]$, with $k$ a field of characteristic zero. This is because Herzog's canonical lift of Koszul cycles \cite{HerCycles} can be applied only under these assumptions.
\begin{question} \label{local} Is there a suitable definition of strongly Golod for local rings, at least when the ring contains a field?
\end{question}
\begin{question} \label{char_p}
Is there a notion of strongly Golod that does not require the characteristic of $k$ to be zero?
\end{question}
\section*{Acknowledgements}
We would like to thank Craig Huneke for stimulating discussions about this project, and for constant encouragement and support. We also thank David Eisenbud for significant help with some Macaulay2 computations and for valuable suggestions. We are indebted to Srikanth Iyeangar, for fruitful conversations that led to Example \ref{Sri}. These took place at the AMS-MRC workshop in Commutative Algebra, Snowbird 2015. Therefore, we would like to thank the American Mathematical Society for providing such a great opportunity to this author. We thank Adam Boocher, Aldo Conca, Lukas Katth{\"a}n and Volkmar Welker for very helpful comments and suggestions. We also thank the anonymous referee, for helpful suggestions that improved the exposition of this article.
\begin{appendices} 
\section{The minimal free resolution of Example \ref{main}} \label{res} Let $k$ be a field, and let $S=k[a,b,c,d,x,y,z,w]$. Consider the monomial ideals $I_1 = (ax,by,cz,dw)$ and $I_2 = (a,b,c,d)$ inside $S$. Let $I:=I_1 I_2$ be their product, and set $R=S/I$. Let $T=\ZZ[a,b,c,d,x,y,z,w]$, and let $J$ be the ideal $I$ inside $T$. Then, using the Macaulay2 command \verb|res J|, we get a resolution of $J$ over $T$
\[
\xymatrixcolsep{5mm}
\xymatrixrowsep{2mm}
\xymatrix{
F_\bullet: 0 \ar[rr] && T^5 \ar[rr]^-{\varphi_4} && T^{20} \ar[rr]^-{\varphi_3} && T^{30} \ar[rr]^-{\varphi_2} && T^{16} \ar[rr]^-{\varphi_1} && T \ar[rr]^-{\varphi_0} && T/I \ar[rr] && 0.
}
\]
Assume that $\Char(k)  = p >0$. We checked with Macaulay2 \cite{Mac2} that $(a^2x) \subseteq I(\varphi_1)$, where $I(\varphi_1)$ is the Fitting ideal of the map $\varphi_1$. This is still a regular element after tensoring with $- \otimes_\ZZ \ZZ/(p)$, so that $\grade(I(\varphi_1 \otimes 1_{\ZZ/(p)})) \gs 1$. Similarly, one can see that 
\[
\ds  \hspace{1cm} (a^{12}x^3,b^{12}y^3) \subseteq I(\varphi_2) \hspace{1cm} (a^{15}x^3,b^{15}y^3,c^{15}z^3) \subseteq I(\varphi_3) \hspace{1cm} (a^5x,b^5y,c^5z,d^5w) \subseteq I(\varphi_4)
\]
Since the former stay regular sequences after tensoring with $-\otimes_\ZZ \ZZ/(p)$ we obtain that 
\[
\ds \grade(I(\varphi_i \otimes 1_{\ZZ/(p)})) \gs i
\]
for all $i =1,\ldots,4$. In addition, the ranks of the maps add up to the correct numbers after tensoring. By Buchsbaum-Eisenbud's criterion for exactness of complexes \cite{BEComplexExact}, $F_\bullet \otimes_\ZZ \ZZ/(p)$ is a minimal free resolution of $J \otimes_\ZZ \ZZ/(p)$ as an ideal of $T \otimes_Z \ZZ/(p) \cong \ZZ/(p)[a,b,c,d,x,y,z,w]$. Finally, since the map $T \otimes_\ZZ \ZZ/(p) \to S$ is faithfully flat, tensoring with $(F_\bullet \otimes_\ZZ \ZZ/(p)) \otimes_{\ZZ/(p)} S$ gives  a minimal free resolution of $I$ over $S$, using Buchsbaum-Eisenbud's criterion for exactness of complexes \cite{BEComplexExact} once again. When $\Char(k) = 0$, one can use $\QQ$ instead of $\ZZ/(p)$ and the same arguments can be applied.

Therefore we get a resolution
\[
\xymatrixcolsep{5mm}
\xymatrixrowsep{2mm}
\xymatrix{
F_\bullet \otimes_\ZZ S: 0 \ar[rr] && S^5 \ar[rr]^-{\varphi_4} && S^{20} \ar[rr]^-{\varphi_3} && S^{30} \ar[rr]^-{\varphi_2} && S^{16} \ar[rr]^-{\varphi_1} && S \ar[rr]^-{\varphi_0} && R \ar[rr] && 0.
}
\]
Letting $E^{(i)}_j$ be the canonical bases of the modules $F_i \cong \bigoplus_{j=1}^{\beta_i} T$, for $i=0,\ldots,4$, the matrices representing the differentials of the minimal free resolution of $J$ over $T$ are the same as the ones of a minimal free resolution of $I$ over $S$. Here follows a description of such matrices. All the missing entries should be regarded as zeros:
\vspace{1cm} 
\setlength{\tabcolsep}{2pt}
\[
\begin{tabular}{cc|c|c|c|c|c|c|c|c|c|c|c|c|c|c|c|c|} $\varphi_1$ && $E^{(1)}_1$ & $E^{(1)}_2$ & $E^{(1)}_3$ & $E^{(1)}_4$ & $E^{(1)}_5$ & $E^{(1)}_6$ & $E^{(1)}_7$ & $E^{(1)}_8$ & $E^{(1)}_9$ & $E^{(1)}_{10}$ & $E^{(1)}_{11}$ & $E^{(1)}_{12}$ & $E^{(1)}_{13}$ & $E^{(1)}_{14}$ & $E^{(1)}_{15}$ & $E^{(1)}_{16}$ \\  \hline &&&&&&&&&&&&&&&&& \\
$E^{(0)}_1$ & & $a^2x$ & $abx$ & $acx$ & $adx$ & $aby$ & $b^2y$ & $bcy$ & $bdy$ & $acz$ & $bcz$ & $c^2z$ & $cdz$ & $adw$ & $bdw$ & $cdw$ & $d^2w$ \\ \hline \end{tabular}
\]
\newpage

\begin{sideways}
\vspace{15cm}
\centering\small\setlength\tabcolsep{1pt}
\begin{tabular}{cc|c|c|c|c|c|c|c|c|c|c|c|c|c|c|c|c|c|c|c|c|c|c|c|c|c|c|c|c|c|c|ccccccccccc} $\varphi_2$ && $E^{(2)}_1$ & $E^{(2)}_2$ & $E^{(2)}_3$ & $E^{(2)}_4$ & $E^{(2)}_5$ & $E^{(2)}_6$ & $E^{(2)}_7$ & $E^{(2)}_8$ & $E^{(2)}_9$ & $E^{(2)}_{10}$ & $E^{(2)}_{11}$ & $E^{(2)}_{12}$ & $E^{(2)}_{13}$ & $E^{(2)}_{14}$ & $E^{(2)}_{15}$ & $E^{(2)}_{16}$ & $E^{(2)}_{17}$& $E^{(2)}_{18}$& $E^{(2)}_{19}$& $E^{(2)}_{20}$ & $E^{(2)}_{21}$ & $E^{(2)}_{22}$ & $E^{(2)}_{23}$ & $E^{(2)}_{24}$ & $E^{(2)}_{25}$ & $E^{(2)}_{26}$ & $E^{(2)}_{27}$& $E^{(2)}_{28}$& $E^{(2)}_{29}$& $E^{(2)}_{30}$ &&&&&&&&&&& \\  \cline{1-32} &&&&&&&&&&&&&&&&&&&&&&&&&&&&&&& \\
$E^{(1)}_1$ & & $-b$ & $-c$ & $$ & $-d$ & $$ & $$ & $$ & $$ & $$ & $$ & $$ & $$ & $$ & $$ & $$ & $$ & $$ & $$ & $$ & $$ & $$ & $$ & $$ & $$ & $$ & $$ & $$ & $$ & $$ & $$\\
\cline{1-32} &&&&&&&&&&&&&&&&&&&&&&&&&&&&&&& \\
$E^{(1)}_2$ & & $a$ & $$ & $-c$ & $$ & $-d$ & $$ & $$ & $$ & $$ & $$ & $$ & $$ & $-y$ & $$ & $$ & $$ & $$ & $$ & $$ & $$ & $$ & $$ & $$ & $$ & $$ & $$ & $$ & $$ & $$ & $$ \\
\cline{1-32} &&&&&&&&&&&&&&&&&&&&&&&&&&&&&&& \\
$E^{(1)}_3$ & & $$ & $a$ & $b$ & $$ & $$ & $-d$ & $$ & $$ & $$ & $$ & $$ & $$ & $$ & $$ & $$ & $$ & $$ & $$ & $$ & $-z$ & $$ & $$ & $$ & $$ & $$ & $$ & $$ & $$ & $$ & $$ \\
\cline{1-32} &&&&&&&&&&&&&&&&&&&&&&&&&&&&&&& \\
$E^{(1)}_4$ & & $$ & $$ & $$ & $a$ & $b$ & $c$ & $$ & $$ & $$ & $$ & $$ & $$ & $$ & $$ & $$ & $$ & $$ & $$ & $$ & $$ & $$ & $$ & $$ & $$ & $$ & $$ & $$ & $-w$ & $$ & $$ \\
\cline{1-32} &&&&&&&&&&&&&&&&&&&&&&&&&&&&&&&\\
$E^{(1)}_5$ & & $$ & $$ & $$ & $$ & $$ & $$ & $-b$ & $-c$ & $$ & $-d$ & $$ & $$ & $x$ & $$ & $$ & $$ & $$ & $$ & $$ & $$ & $$ & $$ & $$ & $$ & $$ & $$ & $$ & $$ & $$ & $$ \\
\cline{1-32} &&&&&&&&&&&&&&&&&&&&&&&&&&&&&&& \\
$E^{(1)}_6$ & & $$ & $$ & $$ & $$ & $$ & $$ & $a$ & $$ & $-c$ & $$ & $-d$ & $$ & $$ & $$ & $$ & $$ & $$ & $$ & $$ & $$ & $$ & $$ & $$ & $$ & $$ & $$ & $$ & $$ & $$ & $$ \\
\cline{1-32} &&&&&&&&&&&&&&&&&&&&&&&&&&&&&&& \\
$E^{(1)}_7$ & & $$ & $$ & $$ & $$ & $$ & $$ & $$ & $a$ & $b$ & $$ & $$ & $-d$ & $$ & $$ & $$ & $$ & $$ & $$ & $$ & $$ & $-z$ & $$ & $$ & $$ & $$ & $$ & $$ & $$ & $$ & $$ \\
\cline{1-32} &&&&&&&&&&&&&&&&&&&&&&&&&&&&&&& \\
$E^{(1)}_8$ & & $$ & $$ & $$ & $$ & $$ & $$ & $$ & $$ & $$ & $a$ & $b$ & $c$ & $$ & $$ & $$ & $$ & $$ & $$ & $$ & $$ & $$ & $$ & $$ & $$ & $$ & $$ & $$ & $$ & $-w$ & $$ \\
\cline{1-32} &&&&&&&&&&&&&&&&&&&&&&&&&&&&&&& \\
$E^{(1)}_9$ & & $$ & $$ & $$ & $$ & $$ & $$ & $$ & $$ & $$ & $$ & $$ & $$ & $$ & $-b$ & $-c$ & $$ & $-d$ & $$ & $$ & $x$ & $$ & $$ & $$ & $$ & $$ & $$ & $$ & $$ & $$ & $$ \\
\cline{1-32} &&&&&&&&&&&&&&&&&&&&&&&&&&&&&&& \\
$E^{(1)}_{10}$ & & $$ & $$ & $$ & $$ & $$ & $$ & $$ & $$ & $$ & $$ & $$ & $$ & $$ & $a$ & $$ & $-c$ & $$ & $-d$ & $$ & $$ & $y$ & $$ & $$ & $$ & $$ & $$ & $$ & $$ & $$ & $$ \\
\cline{1-32} &&&&&&&&&&&&&&&&&&&&&&&&&&&&&&& \\
$E^{(1)}_{11}$ & & $$ & $$ & $$ & $$ & $$ & $$ & $$ & $$ & $$ & $$ & $$ & $$ & $$ & $$ & $a$ & $b$ & $$ & $$ & $-d$ & $$ & $$ & $$ & $$ & $$ & $$ & $$ & $$ & $$ & $$ & $$ \\
\cline{1-32} &&&&&&&&&&&&&&&&&&&&&&&&&&&&&&& \\
$E^{(1)}_{12}$ & & $$ & $$ & $$ & $$ & $$ & $$ & $$ & $$ & $$ & $$ & $$ & $$ & $$ & $$ & $$ & $$ & $a$ & $b$ & $c$ & $$ & $$ & $$ & $$ & $$ & $$ & $$ & $$ & $$ & $$ & $-w$ \\
\cline{1-32} &&&&&&&&&&&&&&&&&&&&&&&&&&&&&&& \\
$E^{(1)}_{13}$ & & $$ & $$ & $$ & $$ & $$ & $$ & $$ & $$ & $$ & $$ & $$ & $$ & $$ & $$ & $$ & $$ & $$ & $$ & $$ & $$ & $$ & $-b$ & $-c$ & $$ & $-d$ & $$ & $$ & $x$ & $$ & $$ \\
\cline{1-32} &&&&&&&&&&&&&&&&&&&&&&&&&&&&&&& \\
$E^{(1)}_{14}$ & & $$ & $$ & $$ & $$ & $$ & $$ & $$ & $$ & $$ & $$ & $$ & $$ & $$ & $$ & $$ & $$ & $$ & $$ & $$ & $$ & $$ & $a$ & $$ & $-c$ & $$ & $-d$ & $$ & $$ & $y$ & $$ \\
\cline{1-32} &&&&&&&&&&&&&&&&&&&&&&&&&&&&&&& \\
$E^{(1)}_{15}$ & & $$ & $$ & $$ & $$ & $$ & $$ & $$ & $$ & $$ & $$ & $$ & $$ & $$ & $$ & $$ & $$ & $$ & $$ & $$ & $$ & $$ & $$ & $a$ & $b$ & $$ & $$ & $-d$ & $$ & $$ & $z$ \\
\cline{1-32} &&&&&&&&&&&&&&&&&&&&&&&&&&&&&&& \\
$E^{(1)}_{16}$ & & $$ & $$ & $$ & $$ & $$ & $$ & $$ & $$ & $$ & $$ & $$ & $$ & $$ & $$ & $$ & $$ & $$ & $$ & $$ & $$ & $$ & $$ & $$ & $$ & $a$ & $b$ & $c$ & $$ & $$ & $$ \\
\cline{1-32} \end{tabular} \hspace*{-1cm}
\end{sideways}
\newpage
\ \ \ 
\vspace{1cm}
\begin{center}
\setlength\tabcolsep{1pt}
\begin{tabular}{cc|c|c|c|c|c|c|c|c|c|c|c|c|c|c|c|c|c|c|c|c|} $\varphi_3$ && $E^{(3)}_1$ & $E^{(3)}_2$ & $E^{(3)}_3$ & $E^{(3)}_4$ & $E^{(3)}_5$ & $E^{(3)}_6$ & $E^{(3)}_7$ & $E^{(3)}_8$ & $E^{(3)}_9$ & $E^{(3)}_{10}$ & $E^{(3)}_{11}$ & $E^{(3)}_{12}$ & $E^{(3)}_{13}$ & $E^{(3)}_{14}$ & $E^{(3)}_{15}$ & $E^{(3)}_{16}$ & $E^{(3)}_{17}$& $E^{(3)}_{18}$& $E^{(3)}_{19}$& $E^{(3)}_{20}$ \\   \hline
$E^{(2)}_1$ && $c$ & $d$ & $$ & $$ & $$ & $$ & $$ & $$ & $$ & $$ & $$ & $$ & $$ & $$ & $$ & $$ & $$ & $$ & $$ & $$ \\  \hline
$E^{(2)}_2$ && $-b$ & $$ & $d$ & $$ & $$ & $$ & $$ & $$ & $$ & $$ & $$ & $$ & $$ & $$ & $$ & $$ & $$ & $$ & $$ & $$ \\  \hline
$E^{(2)}_3$ && $a$ & $$ & $$ & $d$ & $$ & $$ & $$ & $$ & $$ & $$ & $$ & $$ & $$ & $$ & $$ & $$ & $-yz$ & $$ & $$ & $$ \\  \hline
$E^{(2)}_4$ && $$ & $-b$ & $-c$ & $$ & $$ & $$ & $$ & $$ & $$ & $$ & $$ & $$ & $$ & $$ & $$ & $$ & $$ & $$ & $$ & $$ \\  \hline
$E^{(2)}_5$ && $$ & $a$ & $$ & $-c$ & $$ & $$ & $$ & $$ & $$ & $$ & $$ & $$ & $$ & $$ & $$ & $$ & $$ & $-yw$ & $$ & $$ \\  \hline
$E^{(2)}_6$ && $$ & $$ & $a$ & $b$ & $$ & $$ & $$ & $$ & $$ & $$ & $$ & $$ & $$ & $$ & $$ & $$ & $$ & $$ & $-zw$ & $$ \\  \hline
$E^{(2)}_7$ && $$ & $$ & $$ & $$ & $c$ & $d$ & $$ & $$ & $$ & $$ & $$ & $$ & $$ & $$ & $$ & $$ & $$ & $$ & $$ & $$ \\  \hline
$E^{(2)}_8$ && $$ & $$ & $$ & $$ & $-b$ & $$ & $d$ & $$ & $$ & $$ & $$ & $$ & $$ & $$ & $$ & $$ & $xz$ & $$ & $$ & $$ \\  \hline
$E^{(2)}_9$ && $$ & $$ & $$ & $$ & $a$ & $$ & $$ & $d$ & $$ & $$ & $$ & $$ & $$ & $$ & $$ & $$ & $$ & $$ & $$ & $$ \\  \hline
$E^{(2)}_{10}$ && $$ & $$ & $$ & $$ & $$ & $-b$ & $-c$ & $$ & $$ & $$ & $$ & $$ & $$ & $$ & $$ & $$ & $$ & $xw$ & $$ & $$ \\  \hline
$E^{(2)}_{11}$ && $$ & $$ & $$ & $$ & $$ & $a$ & $$ & $-c$ & $$ & $$ & $$ & $$ & $$ & $$ & $$ & $$ & $$ & $$ & $$ & $$ \\  \hline
$E^{(2)}_{12}$ && $$ & $$ & $$ & $$ & $$ & $$ & $a$ & $b$ & $$ & $$ & $$ & $$ & $$ & $$ & $$ & $$ & $$ & $$ & $$ & $-zw$ \\  \hline
$E^{(2)}_{13}$ && $$ & $$ & $$ & $$ & $$ & $$ & $$ & $$ & $$ & $$ & $$ & $$ & $$ & $$ & $$ & $$ & $cz$ & $dw$ & $$ & $$ \\  \hline
$E^{(2)}_{14}$ && $$ & $$ & $$ & $$ & $$ & $$ & $$ & $$ & $c$ & $d$ & $$ & $$ & $$ & $$ & $$ & $$ & $-xy$ & $$ & $$ & $$ \\  \hline
$E^{(2)}_{15}$ && $$ & $$ & $$ & $$ & $$ & $$ & $$ & $$ & $-b$ & $$ & $d$ & $$ & $$ & $$ & $$ & $$ & $$ & $$ & $$ & $$ \\  \hline
$E^{(2)}_{16}$ && $$ & $$ & $$ & $$ & $$ & $$ & $$ & $$ & $a$ & $$ & $$ & $d$ & $$ & $$ & $$ & $$ & $$ & $$ & $$ & $$ \\  \hline
$E^{(2)}_{17}$ && $$ & $$ & $$ & $$ & $$ & $$ & $$ & $$ & $$ & $-b$ & $-c$ & $$ & $$ & $$ & $$ & $$ & $$ & $$ & $xw$ & $$ \\  \hline
$E^{(2)}_{18}$ && $$ & $$ & $$ & $$ & $$ & $$ & $$ & $$ & $$ & $a$ & $$ & $-c$ & $$ & $$ & $$ & $$ & $$ & $$ & $$ & $yw$ \\  \hline
$E^{(2)}_{19}$ && $$ & $$ & $$ & $$ & $$ & $$ & $$ & $$ & $$ & $$ & $a$ & $b$ & $$ & $$ & $$ & $$ & $$ & $$ & $$ & $$ \\  \hline
$E^{(2)}_{20}$ && $$ & $$ & $$ & $$ & $$ & $$ & $$ & $$ & $$ & $$ & $$ & $$ & $$ & $$ & $$ & $$ & $-by$ & $$ & $dw$ & $$ \\  \hline
$E^{(2)}_{21}$ && $$ & $$ & $$ & $$ & $$ & $$ & $$ & $$ & $$ & $$ & $$ & $$ & $$ & $$ & $$ & $$ & $ax$ & $$ & $$ & $dw$ \\  \hline
$E^{(2)}_{22}$ && $$ & $$ & $$ & $$ & $$ & $$ & $$ & $$ & $$ & $$ & $$ & $$ & $c$ & $d$ & $$ & $$ & $$ & $-xy$ & $$ & $$ \\  \hline
$E^{(2)}_{23}$ && $$ & $$ & $$ & $$ & $$ & $$ & $$ & $$ & $$ & $$ & $$ & $$ & $-b$ & $$ & $d$ & $$ & $$ & $$ & $-xz$ & $$ \\  \hline
$E^{(2)}_{24}$ && $$ & $$ & $$ & $$ & $$ & $$ & $$ & $$ & $$ & $$ & $$ & $$ & $a$ & $$ & $$ & $d$ & $$ & $$ & $$ & $-yz$ \\  \hline
$E^{(2)}_{25}$ && $$ & $$ & $$ & $$ & $$ & $$ & $$ & $$ & $$ & $$ & $$ & $$ & $$ & $-b$ & $-c$ & $$ & $$ & $$ & $$ & $$ \\  \hline
$E^{(2)}_{26}$ && $$ & $$ & $$ & $$ & $$ & $$ & $$ & $$ & $$ & $$ & $$ & $$ & $$ & $a$ & $$ & $-c$ & $$ & $$ & $$ & $$ \\  \hline
$E^{(2)}_{27}$ && $$ & $$ & $$ & $$ & $$ & $$ & $$ & $$ & $$ & $$ & $$ & $$ & $$ & $$ & $a$ & $b$ & $$ & $$ & $$ & $$ \\  \hline
$E^{(2)}_{28}$ && $$ & $$ & $$ & $$ & $$ & $$ & $$ & $$ & $$ & $$ & $$ & $$ & $$ & $$ & $$ & $$ & $$ & $-by$ & $-cz$ & $$ \\  \hline
$E^{(2)}_{29}$ && $$ & $$ & $$ & $$ & $$ & $$ & $$ & $$ & $$ & $$ & $$ & $$ & $$ & $$ & $$ & $$ & $$ & $ax$ & $$ & $-cz$ \\  \hline
$E^{(2)}_{30}$ && $$ & $$ & $$ & $$ & $$ & $$ & $$ & $$ & $$ & $$ & $$ & $$ & $$ & $$ & $$ & $$ & $$ & $$ & $ax$ & $by$ \\  \hline

\end{tabular}
\end{center}
\newpage
\ \ 
\vspace{0.6cm}
\begin{center}
\setlength{\tabcolsep}{6pt}
\begin{tabular}{cc|c|c|c|c|c|} $\varphi_4$ && $E^{(4)}_1$ & $E^{(4)}_2$ & $E^{(4)}_3$ & $E^{(4)}_4$ & $E^{(4)}_5$\\  \hline &&&&&& \\
$E^{(3)}_1$ & & $-d$ & $$ & $$ & $$ & $$ \\ \hline
&&&&&& \\
$E^{(3)}_2$ & & $c$ & $$ & $$ & $$ & $$ \\ \hline
&&&&&& \\
$E^{(3)}_3$ & & $-b$ & $$ & $$ & $$ & $$ \\ \hline
&&&&&& \\
$E^{(3)}_4$ & & $a$ & $$ & $$ & $$ & $-yzw$ \\ \hline
&&&&&& \\
$E^{(3)}_5$ & & $$ & $-d$ & $$ & $$ & $$ \\ \hline
&&&&&& \\
$E^{(3)}_6$ & & $$ & $c$ & $$ & $$ & $$ \\ \hline
&&&&&& \\
$E^{(3)}_7$ & & $$ & $-b$ & $$ & $$ & $xzw$ \\ \hline
&&&&&& \\
$E^{(3)}_8$ & & $$ & $a$ & $$ & $$ & $$ \\ \hline
&&&&&& \\
$E^{(3)}_9$ & & $$ & $$ & $-d$ & $$ & $$ \\ \hline
&&&&&& \\
$E^{(3)}_{10}$ & & $$ & $$ & $c$ & $$ & $-xyw$ \\ \hline
&&&&&& \\
$E^{(3)}_{11}$ & & $$ & $$ & $-b$ & $$ & $$ \\ \hline
&&&&&& \\
$E^{(3)}_{12}$ & & $$ & $$ & $a$ & $$ & $$ \\ \hline
&&&&&& \\
$E^{(3)}_{13}$ & & $$ & $$ & $$ & $-d$ & $xyz$ \\ \hline
&&&&&& \\
$E^{(3)}_{14}$ & & $$ & $$ & $$ & $c$ & $$ \\ \hline
&&&&&& \\
$E^{(3)}_{15}$ & & $$ & $$ & $$ & $-b$ & $$ \\ \hline
&&&&&& \\
$E^{(3)}_{16}$ & & $$ & $$ & $$ & $a$ & $$ \\ \hline
&&&&&& \\
$E^{(3)}_{17}$ & & $$ & $$ & $$ & $$ & $-dw$ \\ \hline
&&&&&& \\
$E^{(3)}_{18}$ & & $$ & $$ & $$ & $$ & $cz$ \\ \hline
&&&&&& \\
$E^{(3)}_{19}$ & & $$ & $$ & $$ & $$ & $-by$ \\ \hline
&&&&&& \\
$E^{(3)}_{20}$ & & $$ & $$ & $$ & $$ & $ax$ \\ \hline \end{tabular}
\end{center}
\end{appendices}
\newpage

\bibliographystyle{alpha}
\bibliography{References}
{\footnotesize

\vspace{0.5in}

\noindent \small \textsc{Department of Mathematics, University of Virginia, Charlottesville, VA  22903} \\ \indent \emph{Email address}:  {\tt ad9fa@virginia.edu} 
}
\end{document}